\documentclass[11pt,reqno]{amsart}

\parskip=\smallskipamount

\newtheorem{theorem}{Theorem}[section]
\newtheorem{lemma}[theorem]{Lemma}
\newtheorem{corollary}[theorem]{Corollary}
\newtheorem{proposition}[theorem]{Proposition}

\theoremstyle{definition}
\newtheorem{definition}[theorem]{Definition}
\newtheorem{example}[theorem]{Example}

\newtheorem{remark}[theorem]{Remark}

\newcommand{\bea}{\begin{eqnarray*}}
\newcommand{\eea}{\end{eqnarray*}}

\newcommand{\e}{\mathrm{e}}

\numberwithin{equation}{section}

\begin{document}

\author{Nessim Sibony}
\address{Nessim Sibony, Laboratoire de Math\'ematiques d'Orsay, Univ. Paris-Sud, CNRS, Universit\'e
Paris-Saclay, 91405 Orsay, France. 
Department of Mathematics, University of Oslo, PO-BOX 1053 Blindern, Oslo.}
\author{Erlend Forn\ae ss Wold}
\address{Erlend Forn\ae ss Wold, Department of Mathematics, University of Oslo, PO-BOX 1053 Blindern, Oslo.}
\thanks{The second author is supported by NRC grant number 240569}
\thanks{Part of this work was done during the international research program "Several Complex Variables and Complex Dynamics" at the Centre for Advanced Study at the Academy of Science and Letters in Oslo during the academic year 2016/2017.}
\title[]{Topology and complex structures of leaves of foliations by Riemann surfaces}
%
%

\date{\today}

\begin{abstract}
We study conformal structure and topology of leaves of singular foliations by Riemann surfaces. 
\end{abstract}

\maketitle

\section{introduction}

The question about the topological types of leaves in a lamination has been addressed 
in several important works.  The first striking result is due to Cantwell-Conlon \cite{CantwellConlon2} and 
asserts that any open surface is a leaf of a compact nonsingular lamination.  \

On the other hand Ghys \cite{Ghys2} has considered the following situation.  Let $(X,\mathcal L)$
be a compact non-singular lamination by surfaces, and let $\mu$ be a harmonic 
measure for $X$, as constructed in \cite{Garnett}. Then $\mu$-almost every leaf 
is of one of the following six topological types: a plane, a cylinder, a plane with infinitely many handles attached, a cylinder 
with infinitely many handles attached, a sphere with a Cantor set taken out, or a sphere minus a Cantor 
set with a handle attached to every end.  Ghys uses ergodic theory - Brownian motion with respect to 
$\mu$.  \

Cantwell-Conlon \cite{CantwellConlon} obtained the topological analogue of Ghys' Theorem, \emph{i.e.},
if the lamination is minimal there is a $G_\delta$-dense set of leaves of one of the six types described by Ghys.  \

Here we address the problem of finding conditions for a generic leaf to be a holomorphic disk, and study the 
conformal structure of leaves in a singular foliation.  
We are motivated by a conjecture of Anosov: \emph{for a generic polynomial foliation on $\mathbb P^2$, all 
but countably many leaves are disks} (see e.g. Ilyashenko \cite{Ilyashenko}).  \

Our main concern is holomorphic foliations on $\mathbb P^k$.  Generically, such a foliation does not 
admit a nontrivial directed image of the complex plane, and in particular, all leaves are   
hyperbolic, \emph{i.e.}, they are universally covered by the unit disk.  This allows us to introduce several functions on $X\setminus E$, that for a point $z\in X\setminus E$ measure how much the leaf passing through $z$ looks like the unit disk, observed from the point $z$.  \

For instance, if we let $k(z,v)$ denote the leafwise Kobayashi metric, and $k_\iota(z,v)$ the leafwise \emph{injective}
Kobayashi metric (defined completely analogously to $k$), we set $\rho(z):=k(z,v)/k_\iota(z,v)$.  
Then $0<\rho\leq 1$, 
and $\rho(z)=1$ if and only if the leaf passing through $z$ is the disk (see Section \ref{metrics}
for more details).

\begin{theorem}\label{main}
Let $(X,\mathcal L,E)$ be a Brody hyperbolic holomorphic foliation on a compact complex manifold $X$
of dimension $d=2$, where the 
singular set $E$ is finite.  Assume that there is no compact leaf, and that all singularities are hyperbolic.  Suppose
that there is a sequence $\{z_j\}\subset X\setminus E$ of points such that $\rho(z_j)\rightarrow 1$, 
and $z_j\rightarrow z\in X\setminus E$.  Then there is a nontrivial minimal closed saturated set $Y\subset X$ such that 
all but countably many leaves in $Y$ are disks. 
\end{theorem}
The proof of Theorem \ref{main} uses ergodic theory, and will in fact, granted the assumptions, produce a directed positive $\partial\overline\partial$-closed
current $T$, such that all but countably many leaves in $\mathrm{Supp}(T)$ are disks.  \

Recall that a foliation is Brody hyperbolic if it does not admit a nontrivial holomorphic image of the complex 
plane directed by the foliation, possibly passing through the singularities.  For a generic foliation on $\mathbb P^k$ of degree $d>1$, there is no compact leaf, all singularities 
are hyperbolic, all leaves are hyperbolic, and by Brunella \cite{Brunella}, there is no directed closed current.  In particular 
it is Brody hyperbolic.    \

The function $\rho$ is only lower semicontinuous, so the assumption above does not imply $\rho(z)=1$.
When a singular point $p\in E$ is hyperbolic, there is always a sequence $z_j\rightarrow p$
such that $\rho(z_j)\rightarrow 1$, see Example \ref{hypsing}.  
It is conjectured that for $X=\mathbb P^2$, generically the foliation $(X,\mathcal L,E)$ is minimal, in which case $Y=X$. \

Note that for $X=\mathbb P^2$, generically there is a unique minimal saturated set $\tilde Y\subset X$, due to the unique ergodicity 
theorem proved in \cite{FS1}.   In that case $Y=\tilde Y$ although there is no assumption that $z\in\tilde Y$. \

A consequence of Theorem \ref{main} is that either \emph{all} leaves are far away from resembling 
the disk, or all but countably many leaves are disks.   We have the following dichotomy: 
\begin{theorem}\label{dichotomy}
Let $(X,\mathcal L,E)$ be a Brody hyperbolic holomorphic foliation on a compact complex manifold $X$ of dimension 
$d=2$, where the 
singular set $E$ is finite.  Assume that there is no compact leaf, and that all singularities are hyperbolic.  Then either
\begin{itemize}
\item[(i)] there is a minimal closed saturated set $Y\subset X$ such that 
all but countably many leaves in $Y$ are disks, or
\item[(ii)] the limit set $\Lambda_L$ associated to \emph{any} leaf $L$ is equal to $b\triangle$.
\end{itemize}
\end{theorem}
Recall that if $\Gamma_f$ is the Deck-group associated to a universal covering map $f:\triangle\rightarrow L$, then 
$\Lambda_f$ is the cluster set of $\{\gamma(0)\}_{\gamma\in\Gamma_f}$.  The limit set $\Lambda_L$ of the Riemann surface $L$ is 
well defined modulo conformal transformations of $\triangle$.   In Section 5 we will strengthen this result.  \

We remark that recently, Goncharuk-Kundryashov \cite{GoncharukKundryashov} have constructed 
examples of foliations on $\mathbb P^2$ with the line at infinity invariant, such that all leaves 
have infinitely many handles.  In this case, all leaves except the line at infinity are contained in $\mathbb C^2$.

\section{Metrics and functions on Riemann surface laminations}\label{metrics}

\subsection{The Kobayashi Metric}

Let $L$ be a Riemann surface.  The Kobayashi pseudo-metric $k(z,v)$
is defined as follows.  For $z\in L$ and $v\in T_zL$ we set 
\begin{equation}
k(z,v)=\inf\{\frac{1}{|\lambda|}:\exists f:\triangle\rightarrow L, f(0)=z, f'(0)=\lambda v\},
\end{equation}
where $f$ ranges over all holomorphic maps.  This metric is non-degenerate if and only 
if $L$ is hyperbolic, \emph{i.e.}, if $L$ is universally covered by the unit disk $\triangle$. 
In this case, a universal covering map $f:\triangle\rightarrow L$ is a local isometry with respect to the Poincar\'{e}
metric, \emph{i.e.}, 
\begin{equation}
f^*k(\zeta)=\frac{1}{1-|\zeta|^2}|d\zeta|.
\end{equation}

\subsection{The Injective Kobayashi Metric}

The injective Kobayashi pseudo-metric $k_\iota(z,v)$
is defined as follows.  For $z\in L$ and $v\in T_zL$ we set 
\begin{equation}
k_\iota(z,v)=\inf\{\frac{1}{|\lambda|}:\exists f:\triangle\rightarrow L, f(0)=z, f'(0)=\lambda v\},
\end{equation}
where $f$ ranges over all \emph{injective} holomorphic maps.

 \

\subsection{The Suita Metric}

Assume for a moment that the Riemann surface $L$ is hyperbolic in the sense of Ahlfors, \emph{i.e.}, that 
$L$ carries Green's functions.  For a point $z\in L$ we let $G_z(x)$ denote the negative
Green's function with pole at $z$.  In a local coordinate system $w$ with $w(z)=0$ we may write 
\begin{equation}
G_z(w)=\log |w| + h_z(w), 
\end{equation}
with $h_z$ harmonic.  In local coordinates $z$ on $L$ the 
Suita metric $s(z,v)$ is defined by $c_\alpha(z)|dz|=\exp(h_z(0))|dz|$. 
If $L$ does not support a Green's function we set $c_\alpha\equiv 0$.

\subsection{Some functions on Riemann surfaces.}

We now define some functions with values in $[0,1]$, and which 
achieve the value one at a point $z\in L$ if and only if $L$
is the unit disk.  We let $f:\triangle\rightarrow L$ be a universal 
covering map with $f(0)=z$, and we let $\Gamma$ denote the 
associated Deck-group.

\begin{equation}
\rho(z):=\frac{k(z,v)}{k_\iota(z,v)},
\end{equation}
\begin{equation}
\alpha(z):=\frac{s(z,v)}{k(z,v)},
\end{equation}
\begin{equation}
\beta(z):=\min\{|\gamma(0)|:\gamma\in\Gamma\}.
\end{equation}
These functions have concrete geometric interpretations.  Note first that 
by Hurwitz Theorem and a normal family argument there exists an 
injective holomorphic map $g:\triangle\rightarrow L, g(0)=z,$ that realises
$k_\iota$, and a
universal covering map $f$
realises $k$.  Then $g$ can be factored through $f$, \emph{i.e.}, 
there exists an injective holomorphic $h:\triangle\rightarrow\triangle, h(0)=0$, 
such that $f(h(z))=g(z)$.   By the chain rule we see that 
\begin{equation}\label{intrho}
\rho(z)=h'(0).
\end{equation}

For an interpretation of $\alpha$ we have by Myrberg's 
theorem \cite{Myrberg},\cite{Tsuji}  that 
\begin{equation}
G_z(f(\zeta))=\sum_{\gamma\in\Gamma}\log|\frac{\zeta-\gamma(0)}{1-\overline{\gamma(0)}\zeta}|.
\end{equation}
In the coordinate system given by $f$ we have that $k(z,v)=|v|$, and it follows that
\begin{equation}\label{suitacover}
\alpha(z)=\Pi_{\gamma\neq\mathrm{id}}|\gamma(0)|.
\end{equation}
So $\alpha(z)$ is the product of the M\"{o}bius lengths of the shortest elements in each homotopy class
based at $z$.  Finally, $\beta(z)$ is the shortest length occuring in this product, \emph{i.e.}, the 
length of the shortest non-trivial loop based at $z$.

\begin{example}
Let $Y$ be a compact Riemann surface, and let $\Omega\subset Y$ be a domain.
We define the metrics on the Riemann surface $\Omega$.
Set $K=Y\setminus\Omega$, and let  
$K_1\subset K$ be a compact set with $K\setminus K_1$ closed.  Assuming
that $K_1$ has logarithmic capacity zero, then $\lim_{z\rightarrow K}\alpha(z)=0$ (the simplest
example would be if $K_1$ is an isolated point).  Assuming that 
$\lim_{z\rightarrow K_1}\beta(z)=0$ we will also have $\rho(z),\alpha(z)\rightarrow 0$.
Convergence of $\alpha(z)$ is clear since $\alpha<\beta$.  To see the convergence 
of $\rho$ we let $z_j\rightarrow z_0\in K_1$, we let $f_j:\triangle\rightarrow\Omega$
be a universal covering map with $f_j(0)=z$, and we let $g_j:\triangle\rightarrow\Omega$
be injective holomorphic with $h_j(0)=z$ and $\kappa_\iota(z)=|h'(0)|^{-1}$.  Let 
$h_j:\triangle\rightarrow\triangle$ factor $g_j$ through $f_j$, \emph{i.e.}, we have that 
$g_j=f_j\circ h_j$, so that $|h'_j(0)|=\rho(z_j)$.  Since $\beta(z_j)\rightarrow 0$ 
we see that $f_j$ cannot be injective on disks of radius $r_j$ where $r_j\rightarrow 0$, 
and so by Lemma \ref{SCH} we have that $\rho(z_j)\rightarrow 0$.
\end{example}

We have further the following relations between the functions.

\begin{proposition}\label{implications}
Let $L$ be a hyperbolic Riemann surface and let $\{z_j\}\subset L$ be a sequence of points. 
Then 
\begin{equation}
\alpha(z_j)\rightarrow 1\Rightarrow \beta(z_j)\rightarrow 1\Leftrightarrow\rho(z_j)\rightarrow 1.
\end{equation}
Moreover, if $g$ denotes any of the three functions, the following holds: For any 
$\delta>0$ (small) and $R>0$ (large) there exists $\epsilon>0$ such that 
if $L$ is any Riemann surface, and $z\in L$ with $g(z)\geq 1-\epsilon$, then 
$g(y)\geq 1-\delta$ for all $y\in L$ with $\mathrm{d_K}(z,y)\leq R$.
\end{proposition}
\begin{proof}
The first implication is clear since $\alpha<\beta$.  For the second right implication,
fix a universal covering map $f:\triangle\rightarrow L$ with $f(0)=z$, and 
fix $\gamma\in\Gamma$ such that $\beta(z)=|\gamma(0)|$.  Set 
\begin{equation}
b=\frac{1}{2}\log(\frac{1+\beta(z)}{1-\beta(z)}),
\end{equation}
\emph{i.e.}, the Kobayashi length from $0$ to $\gamma(0)$.  By the triangle inequality $\Gamma$
cannot identify points in a disk of Kobayashi radius $b/3$, hence $f$ is injective 
on the disk centred at the origin of radius
\begin{equation}
i(\beta(z)):=\frac{(1+\beta(z))^{\frac{1}{3}}-(1-\beta(z))^{\frac{1}{3}}}{(1+\beta(z))^{\frac{1}{3}} + (1-\beta(z))^{\frac{1}{3}}}.
\end{equation}
So 
\begin{equation}\label{estrho}
\rho(z)\geq i(\beta(z)).
\end{equation}
For the last implication, we again fix a universal covering map at a point $z$, and we fix an injective 
$h:\triangle\rightarrow\triangle$ with $h(0)=0$ such that $\rho(0)=h'(0)$.  By the following lemma
we see that 
\begin{equation}\label{estbeta}
\beta(z)\geq (\frac{1-\sqrt{1-\rho(z)^2}}{\rho(z)})^2
\end{equation}

\begin{lemma}\label{SCH}
Let $h:\triangle\rightarrow\triangle$ be a holomorphic map with $h(0)=0$, and set
$\lambda=|h'(0)|$.  Then $\triangle_r\subset h(\triangle)$ with 
\begin{equation}
r=(\frac{1-\sqrt{1-\lambda^2}}{\lambda})^2
\end{equation}
\end{lemma}
\begin{proof}
Assume that $-r\notin h(\triangle), r>0$, and set $\phi(\zeta)=\frac{\zeta+r}{1+r\zeta}$, and then $g(\zeta)=\phi(h(\zeta))$.
Then $g(0)=r$.  Let $f$ the square root of $g$ such that $f(0)=\sqrt r$, set $\psi(\zeta)=\frac{\zeta-\sqrt r}{1-\sqrt r\zeta}$, 
and then $q(\zeta)=\psi(f(z))$.  Now $g'(0)=(1-r^2)\lambda$, and since $(f\cdot f)'(0)=2f(0)f'(0)$ and $\psi'(\sqrt r)=1/(1-r)$
we get that 
\begin{equation}
q'(0)=\frac{(1+r)\lambda}{2\sqrt r}\leq 1\Leftrightarrow \lambda r - 2\sqrt{r} + \lambda <0,
\end{equation}
by Schwarz Lemma.   The expression on the right is zero when 
\begin{equation}
\sqrt{r}=\frac{2\pm\sqrt{4-4\lambda^2}}{2\lambda}.
\end{equation}
\end{proof}
Finally we consider the last claim.  If $g$ is equal to $\rho$ or $\beta$, it suffices by \eqref{estrho} and \eqref{estbeta}
to prove the claim for either of them.   For $g=\beta$ this is a simple consequence of the triangle inequality.  \

For $g=\alpha$
we fix $x\in L$ and let $f:\triangle\rightarrow L$ be a universal covering map.   Then by \eqref{suitacover}
\begin{equation}
\alpha(f(\zeta))=\Pi_{\gamma\neq\mathrm{id}}|\frac{\zeta-\gamma(\zeta)}{1-\overline{\gamma(\zeta)}\zeta}|=\Pi_{\gamma\neq\mathrm{id}}\mathrm{d_M}(\zeta,\gamma(\zeta)),
\end{equation}
where $\mathrm{d_M}$ denotes the M\"obius distance on $\triangle$.  Letting $\mathrm{d_P}$ denote 
the Poincar\'{e} distance on $\triangle$ this can be rewritten as
\begin{equation}
\alpha(f(\zeta)) = \Pi_{\gamma\neq\mathrm{id}}\frac{\e^{2\mathrm{d_P}(\zeta,\gamma(\zeta))}-1}{\e^{2\mathrm{d_P}(\zeta,\gamma(\zeta))}+1}.
\end{equation}
If we set $r=\frac{\e^{2R}-1}{\e^{2R}+1}$ we have that $|\zeta|<r$ for all $\zeta$ with $f(\zeta)=y$ with $\mathrm{d_P}(x,y)<R$.
By the triangle inequality we have that 
\begin{equation}
\mathrm{d_P}(\zeta,\gamma(\zeta))\geq\mathrm{d_P}(0,\gamma(\zeta)) - \mathrm{d_P}(0,\zeta)\geq \mathrm{d_P}(0,\gamma(\zeta))-R,
\end{equation}
and furthermore
\begin{equation}
\mathrm{d_P}(0,\gamma(\zeta))\geq\mathrm{d_P}(0,\gamma(0))-R.
\end{equation}

It follows that 
\begin{align*}
\alpha(f(\zeta)) & \geq \Pi_{\gamma\neq\mathrm{id}}\frac{\e^{2\mathrm{d_P}(0,\gamma(0))}\e^{-4R}-1}{\e^{2\mathrm{d_P}(0,\gamma(0))}\e^{-4R}+1} \\
& = \Pi_{\gamma\neq\mathrm{id}}(1 - \frac{2}{\e^{2\mathrm{d_P}(0,\gamma(0))}\e^{-4R}+1})
\end{align*}
Fix $\tilde\delta>$ such that $\alpha(f(\zeta))>1-\delta$ if $\log\alpha(f(\zeta))>-\tilde\delta$.  Now
\begin{equation}
\log(\alpha(f(\zeta)))\geq \sum_{\gamma\neq\mathrm{id}}-\frac{4}{\e^{2\mathrm{d_P}(0,\gamma(0))}\e^{-4R}+1}\geq\sum_{\gamma\neq\mathrm{id}}-\frac{8\e^{4R}}{\e^{2\mathrm{d_P}(0,\gamma(0))}+1}
\end{equation}
if $\epsilon>0$ is small enough.   On the other hand 
\begin{equation}
\log\alpha(f(0))=\sum_{\gamma\neq\mathrm{id}}\log(1- \frac{2}{\e^{2\mathrm{d_P}(0,\gamma(0))}+1})\leq \sum_{\gamma\neq\mathrm{id}}- \frac{1}{\e^{2\mathrm{d_P}(0,\gamma(0))}+1},
\end{equation}
and so
\begin{equation}
\log\alpha(f(\zeta))\geq 8e^{4R}\log\alpha(f(0)),
\end{equation}
which is greater than $-\tilde\delta$ if $\epsilon>0$ is small enough.  
\end{proof}

The values of the functions $\alpha,\beta$ and $\rho$ at a point $z$ on a Riemann surface $X$, can be regarded as measuring how much $L$ resembles the unit disk observed from the point $z$.  Indeed, each of them take values in the unit interval, 
and $\alpha(z)=1$ for a point $z\in L$ if and only if $L$ is biholomorphic to the unit disk.   On the other hand, there
are many Riemann surfaces $L$ for which the quantity 
\begin{equation}
s_\alpha(L):=\sup_{z\in L}\{\alpha(z)\}
\end{equation}
is equal to one.  For instance we have the following. 

\begin{lemma}
Let $\Gamma$ be a Fuchsian group such that the limit set $\Lambda(\Gamma)$ is different from 
$b\triangle$, and let $L=\triangle/\Gamma$ denote the underlying Riemann surface.  Then $s_\alpha(L)=1$.
\end{lemma}
\begin{proof}
Fix $\theta$ such that $e^{i\theta}\notin\Lambda(\Gamma)$.  By \eqref{estlog} below there exists 
a constant $C>0$ such that $\alpha(f(re^{i\theta}))\geq C(1-r)$, where $f:\triangle\rightarrow X$
denotes the universal covering map (in fact the constant $C$ depends only on the distance 
to the limit set).  
\end{proof}
We remark that when $\Lambda(\Gamma)\neq b\triangle$, then $\Gamma$ is of convergence 
type, or equivalently, $\triangle/\Gamma$ supports a non-trivial bounded subharmonic function. 
The reason is that if $p\in b\triangle\setminus\Lambda(\Gamma)$, the group $\Gamma$
cannot identify points near $p$, and so it is easy to construct $\Gamma$-invariant 
subharmonic functions. \

In the final section we will construct a further example where $\Lambda(\Gamma)=b\triangle$
but still $s_\alpha(L)=1$ where $L=\triangle/\Gamma$.  \
\subsection{Riemann surface laminations}

We will be interested in the metrics and functions defined above in the setting of laminations by Riemann surfaces.     
Recall that a non-singular lamination $(X,\mathcal L)$ in a complex manifold $M$ is a closed
subset $X\subset M$ such that for each point $p\in X$, there are local coordinates $\phi(x)=(z,w)\in\triangle\times\triangle^{n-1}$
near $p$, and a closed subset $T\in\triangle^{n-1}$ such that $\phi(U_p\cap X)$ is a disjoint union of holomorphic graphs
$(z,g_t(z))$ with $g_t(0)=t\in T$.   Moreover $g_t$ varies continuously with $t$ (the last assumption is unnecessary if $n=2$ in which case
$g_t$ is automatically almost Lipschitz).   The concept of a lamination generalises to that of an abstract lamination.  An abstract 
lamination by Riemann surfaces $(X,\mathcal L)$ is a locally compact topological space $X$ covered by charts $U_i$ with embeddings 
$\phi_i:U_i\rightarrow\triangle\times T_i$, and continuous transition mappings 
\begin{equation}
(z,t)\rightarrow (z'(z,t),t'(t)),
\end{equation}
with $z'(z,t)$ holomorphic in $z$.  Depending on the transversals $T_i$ one can also consider higher transverse 
regularity.    Finally a singular Riemann surface lamination $(X,\mathcal L,E)$  is a compact topological space $X$ with $E\subset X$
a finite set of points, $X\setminus E$ is a non-singular lamination, and for each point $p\in E$ there exists an 
open neighbourhood $U_p$ of $p$, and a homeomorphism $\phi_p$ from $U_p$ onto a 
closed set $Y\subset \mathbb B^n$ with $\phi_p(p)=0$, where $Y\setminus\{0\}$ is a non-singular lamination, and 
$\phi_p$ is holomorphic along leaves. \

From now on we will consider compact Riemann surface laminations $(X,\mathcal L,E)$.  
Outside of $E$ we have that $X$ can be equipped with a leafwise hermitian metric $\omega$,
to obtain a refined structure $(X,\mathcal L,E,\omega)$.  Near a singular point $p\in E$ we will always assume that such a metric is comparable to $\phi_p^*\omega_E$, where 
$\omega_E$ is the euclidean metric.  \

The main examples we have in mind are laminated sets in compact complex manifolds, in particular in $\mathbb P^2$, and laminations 
constructed as suspensions or towers of compact Riemann surfaces, see e.g. \cite{FS2}, \cite{FSW}, \cite{Ghys} and 
references in there.  \

From now on we assume that all leaves of a lamination are hyperbolic.  
Then the Kobayashi metrics, Suita metric, and the functions
$\rho,\alpha,\beta$ can be defined along the leaves of $(X,\mathcal L,E)$.  Recall that the Suita metric is 
set to be zero on a Riemann surface that does not support a Green's function.  However, 
a leaf not supporting 
a Green's function would give rise to a positive closed current \cite{PaunSibony}.  So by \cite{Brunella}, 
for a generic foliation on $\mathbb P^k$ of degree $d>1$, the Suita metric is non-degenerate on all leaves.

 \begin{lemma}
Let $(X,\mathcal L,E)$ be a compact hyperbolic Riemann surface lamination. Assume that there is no non-constant holomorphic map $f:\mathbb C\rightarrow X$ weakly directed by $\mathcal L$, and assume that all singularities are hyperbolic.  
Then the following holds:
\begin{itemize}
\item[(i)] $\rho$ is lower semi-continuous, and continuous on all leaves without holonomy. 
\item[(ii)] $\alpha$ is upper semi-continuous on all leaves without holonomy.
\item[(iii)] $\beta$ is lower semi-continuous, and continuous on leaves without holonomy.  
\end{itemize}

\end{lemma}
\begin{proof}
First, lower semi-continuity in (i) follows from Proposition \ref{productstructure} since $k(z,v)$ is continuous, and injective holomorphic 
maps will lift to nearby leaves.   Next, assume to get a contradiction that there is a point $z_0$ on a leaf $L_0$
without holonomy at which $\rho$ is not upper semicontinuous, \emph{i.e.}, $k_\iota(z,v)$ is not lower semi-continuous. 
Then there exists a sequence $z_j\subset L_j$ with $z_j\rightarrow z_0$, and $\lim_{j\rightarrow\infty}k_\iota(z_j,v)<k_\iota(z_0,v)$. 
If we let $f_j:\triangle\rightarrow L_j$ realise $k_\iota$ at $z_j$ for $j=0,1,2,3,...,$ this means that 
$\lim_{j\rightarrow\infty} |f_j'(0)|>|f_0'(0)|$ (evaluated in some local coordinates).  Since $\{f_j\}$ is a
normal family, we may assume that $f_j\rightarrow\tilde f_0$ uniformly on compacts.   Then $|\tilde f_0'(0)|>|f_0'(0)|$
and so $\tilde f_0$ cannot be injective.  Choose distinct points $a,b\in\triangle$ such that $\tilde f_0(a)=\tilde f_0(b)$, 
and let $l$ be the straight line segment between $a$ and $b$.  Then $\tilde f_0(l)$ is a closed loop 
in $L_0$, and $\tilde f_j(l)$ would determine a lifting of this loop to an open curve for $j$ large, a contradiction
to the fact that $L_0$
is without holonomy, \emph{i.e}, a any compact in $L_0$ has a fundamental neighborhood system with product structure
(see \cite{EpsteinMillettTischler}, \cite{Hector} and Proposition \ref{productstructure} below). 

\

To show (ii),
fix a point $w\in L$ where $L$ is a leaf without holonomy.  Since $k(z,v)$ is continuous it suffices to show that $c_\beta$ is upper semi-continuous at $w$.  
For $\epsilon>0$ choose
a smooth domain $Y\subset L$ with $w\in Y$ such that $c_{\alpha,Y}(w)<c_\alpha(w)+\epsilon$.
By Proposition \ref{productstructure} the foliation has a product structure $Y\times T$
near $Y$ and so for any leaf $L_t$ near $L$ there is a domain $Y_t\subset L_t$
such that $c_{\alpha,Y_t}(w_t)<c_\alpha(w)+2\epsilon$ for $w_t$ close to $w$.
Since the Suita metric is decreasing with respect to increasing domains, this gives
the upper semi-continuity of $\alpha$.   \

To show (iii), note first that $\beta$
is lower semi-continuous by the continuity of the Kobayashi metric and continuity of 
the universal covering maps after appropriate normalisation.   For the last
claim, fix a point $w$ on a leaf $L_t, t\in T$, without holonomy, and let $f:\triangle\rightarrow L_t$
be a universal covering map.   Let $\epsilon>0$ be small, and let $Y\subset L_t$ be a 
smooth domain such that $f(\triangle_{|\gamma(0)|+\epsilon})\subset Y$.  By Proposition \ref{productstructure}
there is a product structure near $Y$ and since the universal covering maps may be 
chosen to vary continuously, there are sequences of points $a_j\rightarrow 0$ and $b_j\rightarrow\gamma(0)$
such that $f_{t_j}(a_j)=f_{t_j}(b_j)$ when $t_j\rightarrow t$.  This means that there are elements 
$\gamma_j$ in the Deck-groups such that $\gamma_j(a_j)=b_j$, hence $\gamma_j(0)\rightarrow\gamma(0)$.
\end{proof}

\begin{proposition}
Let $(X,\mathcal L,E)$ be a foliation on $\mathbb P^n$ , and assume that all singularities are hyperbolic.  
Then there exists a $\delta>0$ such that $\beta(z)\geq\delta$ if $z\in\mathbb P^n\setminus E$, unless
$z$ is on a separatrix and is close to $E$.  
\end{proposition}
\begin{proof}
For such a foliations all leaves are hyperbolic, and the Kobayashi metric is continuous (see \cite{CandelGomez-Mont} 
and the survey \cite{FS2}).  
Near any point $p\in E$, all leaves except finitely many separatrices are simply connected. 
So in local coordinates where $p=0$, there are $0<\delta_1<\delta_2<<2$ such that 
if $z\in B_{\delta_1}(0)$, not on a separatrix, then any nontrivial loop based at $z$ 
will have to leave $B_{\delta_2}(0)$, and so the length of such a curve is bounded away from zero.  
 \end{proof}

\begin{proposition}
Suppose $(X,\mathcal L,E)$ is a compact minimal Brody hyperbolic Riemann surface lamination, and 
assume that all singularities are hyperbolic.   
Assume that one leaf is a disk.  Then a generic leaf is a disk, and for any leaf $L$ there 
exists a sequence $z_j\in L$ with $z_j\rightarrow z_0\notin E$, and 
\begin{equation}
\lim_{j\rightarrow\infty}\beta(z_j)=\lim_{j\rightarrow\infty}\rho(z_j)=1.
\end{equation}
\end{proposition}
\begin{proof}
For each $n$ we have that the set $U_n=\{\rho>1-1/n\}$ is open by lower semi-continuity of $\rho$,
and by minimality we have that $U_n$ is dense.  So $\cap_n U_n$ is a dense 
$G_\delta$ set on which $\rho\equiv 1$, and so the corresponding 
leaves are disks.   Furthermore, any leaf $L$ will have to cluster onto a disk away from $E$, 
and lower semi-continuity implies (2.28).
\end{proof}

We may now also define the functions $s_{\alpha,\beta,\rho}(X)$ 
on a Riemann surface lamination $X$; we simply take the suprema over all leaves.  

\begin{proposition}
Suppose $(X,\mathcal L,E)$ is a compact Brody hyperbolic Riemann surface lamination, and 
assume that all singularities are hyperbolic.  
Then if $L$ is a dense leaf we have that $s_\rho(X)=s_\rho(L)$ and 
$s_\beta(X)=s_\beta(L)$.
\end{proposition}
\begin{proof}
This follows by lower semi-continuity of the functions $\rho$ and $\beta$.
\end{proof}

\section{Product structures on laminations by complex manifolds}

In this section we give some basic results about holomorphic maps 
into holomorphic foliations, and product structures on
leaves without holonomy.   As observed in \cite{FS2}, for foliations on complex manifolds, these 
results follows by a construction due to Royden \cite{Royden}; our emphasis 
here is on abstract laminations.

\begin{definition}
A complex manifold $M$ is a \emph{Stein manifold} if $M$ admits
a strictly plurisubharmonic exhaustion function $\rho$.  A compact set $K\subset M$ of a complex manifold $M$ is a \emph{Stein compact} if it has a fundamental system of open Stein neighbourhoods.  
\end{definition}

\begin{definition}
A pair $(A,B)$ of compact subsets in a complex manifold $X$ is a Cartan pair if
the following holds
\begin{itemize}
\item[(i)] $A,B, D=A\cup B$ and $C=A\cap B$ are Stein compacta, and
\item[(ii)] $\overline{A\setminus B}\cap\overline{B\setminus A}=\emptyset$.
\end{itemize}
\end{definition}

\begin{proposition}\label{productstructure}
Let $(X,\mathcal L)$ be a lamination by complex manifolds of dimension $d$ on a metric space 
$X$.    Fix a local transversal $T$ and a point $t_0$ in $T$.  Let $M$ be a Stein manifold 
with a strictly plurisubharmonic Morse exhaustion function $u$, and let $M_c=\{u\leq c\}$
be a smooth simply connected sublevel set. 
Then for any holomorphic immersion $f:M_c\rightarrow L_{t_0}$
with $t_0\in f(M_c)$, there exists an open neighbourhood $T'\subset T$ of $t_0$
and a continuous map $F:M_c\times T'\rightarrow X$ such that 
\begin{itemize}
\item[(i)] $F_t:M_c\times\{t\}\rightarrow L_{t}$ is holomorphic for all $t$, and 
\item[(ii)] $F_{t_0}=f$.
\end{itemize}
Furthermore, if $Y_0\subset L_{t_0}$ is a strictly psedoconvex domain with boundary and without holonomy
and $t_0\in Y_0$, there exists an open neighbourhood $T'\subset T$ of $t_0$ and a continuous 
map $F:Y_0\times T'\rightarrow X$ such that 
\begin{itemize}
\item[(iii)] $F_t:Y_0\times\{t\}\rightarrow L_{t}$ is a holomorphic embedding for all $t$, and 
\item[(iv)] $F_{t_0}=\iota$, where $\iota$ is the inclusion map. 
\end{itemize}
\end{proposition}

\begin{proof}
Let $A_k$ be a finite sequence of compact strongly pseudoconvex domains with 
$A_{k+1}=A_k\cup B_k$, such that $(A_k,B_k)$ is a Cartan pair for each $k$, and 
$M_c=\cup_k A_k$.    If $\{s_j\}, j=1,...,m,$ are the singular values of $u$ less than $c$, and $0<\epsilon<<1$,
each sublevel set $\{u\leq s_j-\epsilon\}$ will occur as $A_k$ for some $k$, and in that
case $B_k$ will be a topological $p$-cell, $p$ being the Morse index of $u$ at the critical point, such 
that $A_{k+1}$ is diffeomorphic to $\{u\leq s_j+\epsilon\}$.  For any other $A_k$, the set $B_k$
will be a small "bump" on $A_k$, such that $A_k\cap B_k$ is connected, and $A_{k+1}$ is 
diffeomorphic to $A_{k}$.
Furthermore, we ensure that 
$f(A_1)$ is contained in the flow box $\mathbb B^d\times T_1, T_1=T$, 
$f(B_k)$ is contained in a flow box $\mathbb B^d\times T_k$ for all $k$, and that $f$ is
injective on $A_1$ and on  each $B_k$ (see e.g.
\cite{Forstnericbook}, 3.10 and 5.10 for details on the existence of such a family of bumps).  \

We will prove, by induction on $k$, that there are transversals $T^k\subset T$ such that 
(i) and (ii) holds with $M_c$ replaced by $A_k$ and $T'$ replaced by $T^k$.  This is clearly 
the case for $A_1$ and $T^1=T_1$, since we can lift $f$ inside the flowbox. \

Assume now that (i) and (ii) hold for the pair $(A_k,T^k)$, \emph{i.e.}, we have constructed 
a map $F_k:A_k\times I^k\rightarrow X$ satisfying (i) and (ii).   First we will construct 
a local lifting $G_{k+1}$ of $f:B_{k}\rightarrow\mathbb B^d\times T_{k+1}$.  For each 
$t\in T^{k}$ the map $G_{k+1}:A_{k}\cap B_{k}\rightarrow X$ will determine 
which leaf we should lift to in the flow box, and then we lift using the projection in the 
flow box.  This is where simply connectedness is used in the case where attaching 
$B_k$ corresponds to attaching a 1-cell.  The maps $F_k$ and $G_{k+1}$ do not match, but $\gamma_t(\cdot)=G_{k+1}^{-1}\circ F_{k}(\cdot,t)$ converges to the identity on a neighbourhood of $B_{k}$
as $t\rightarrow t_0$.  By Theorem 8.7.2 in \cite{Forstnericbook} there exist an open set
$U\supset A_k$ and an open subset $V\supset B_k$ such that for $t$ close
enough to $t_0$, there are continuous families of injective holomorphic maps $\alpha_t:U\rightarrow M, \beta_t:V\rightarrow M$
such that $\beta_t\circ\alpha_t^{-1}=\gamma_t$ near $A_k\cap B_k$, and $\alpha_{t_0}=\beta_{t_0}=\mathrm{id}$.
Then the maps $F_t\circ\alpha_t$ and $G_t\circ\beta_t$ fit together near $A_k\cap B_k$ to form a map $F_{k+1}(\cdot,t)$
as long as $T$ belongs to a transversal $T^{k+1}$ contained in a small neighbourhood of $t_0$. \

Finally, the existence of a map $F$ satisfying (iii) and (iv) is proved in exactly the same way, inductively constructing liftings of 
$A_k$ where $(A_k,B_k)$ is a family of "bumps" on $Y_0$; the absence of holonomy makes sure that the leafs
match when attaching $B_k$ corresponds to crossing a singular point of Morse index one.   
\end{proof}

\section{Proof of Theorem \ref{main}}

We will now prove  Theorem \ref{main}.  

\begin{theorem}\label{main2}
Let $(X,\mathcal L,E)$ be a be a Brody hyperbolic holomorphic foliation on a compact complex manifold $X$
of dimension $d=2$, where the 
singular set $E$ is finite,  Assume that there is no compact leaf, and that all singularities are hyperbolic.
Suppose
that there is a sequence $\{z_j\}\subset X\setminus E$ of points such that $\rho(z_j)\rightarrow 1$, 
and $z_j\rightarrow z\in X\setminus E$.  Then there is a nontrivial closed minimal saturated set $Y\subset X$ such that 
all but countably many leaves in $Y$ are disks.  \
\end{theorem}

Note that $E$ might be empty, in which case the assumption $z_j\rightarrow z$ is unnecessary. 
\begin{remark}
It is seen from the proof below, that if we add some extra conditions, similar results hold
also for $d>2$, and for for compact Riemann surface laminations
$(X,\mathcal L,E)$.   For a holomorphic foliation $(X,\mathcal L,E)$ 
we add the condition that there is no positive directed $\partial\overline\partial$-closed
current $T$, whose support contains only leaves with holonomy in $(X,\mathcal L,E)$.
For a compact Riemann surface lamination
$(X,\mathcal L,E)$, we 
assume in addition that the lamination is minimal, in which case $Y=X$, and 
replace "all but countably many leaves" by "a residual set of leaves".
\end{remark}

\begin{proof}
Let $f:\triangle\rightarrow X$ be a universal covering map with $f(0)=z$.
For $0<r<1$ we define a $(1,1)$-current $G_r$ on $\triangle$ by setting 
\begin{equation}
\langle G_r,\alpha\rangle := \int\int_\triangle\log^+\frac{r}{|\zeta|}\alpha, 
\end{equation}
and we set $T_r=f_*G_r$.  Then $T_r$ is a positive current of bidimension $(1,1)$.  By Theorem 5.3 in \cite{FS3} we have that $\|T_r\|\rightarrow\infty$
as $r\rightarrow 1$, \emph{i.e.}, 
\begin{equation}
\lim_{r\rightarrow 1}\int\int_\triangle\log^+\frac{r}{|\zeta|}f^*\omega=\infty. 
\end{equation}
By Proposition \ref{implications} we may choose a sequence $k(j)$ such that if $f_j:\triangle\rightarrow X$
is a universal covering map of the leaf passing through $z_{k(j)}$ sending $0$ to $z_{k(j)}$, 
then $\beta(f_j(\zeta))>1-1/j$ for all $|\zeta|<1-1/j$.  Furthermore, $k(j)$ may be chosen such that $f_j$
approximates $f$ arbitrarily well on $\triangle_{1-1/j}$, and so we have that
\begin{equation}
\lim_{j\rightarrow\infty}\int\int_\triangle\log^+\frac{1-1/j}{|\zeta|}f_j^*\omega=\infty. 
\end{equation}
So if we set $T_j:=f_{j*}G_{1-1/j}$ we have that $\|T_j\|\rightarrow\infty$ as $j\rightarrow\infty$ and 
that $\beta(z)>1-1/j$ for all $z$ in the support of $T_j$.  Let $T$ be any cluster point 
of $\{T_j/\|T_j\|\}$.  Then $T$ has mass one.  Since the masses of $\partial\overline\partial T_j$ are uniformly bounded 
the current $T$ is $\partial\overline\partial$-closed, and so its support is a sub-foliation of $(X,\mathcal L,E)$.
The current $T$ admits in a flow-box a decomposition 
\begin{equation}
\langle T,\omega\rangle=\int (\int_{\triangle_t}h_t\omega) d\mu(t)
\end{equation}
where $h_t$ is harmonic for each $t$ on the transversal, and since the foliation has no closed leaf, the 
measure $\mu$ is diffuse.  \

Let $Y\subset X$ be a minimal foliation on $\mathrm{Supp}(T)$.  Then, since $Y$ is not a single compact leaf, there are 
uncountably many leaves in $Y$.  Moreover, the leaves with holonomy in $(X,\mathcal L,E)$ form a 
countable set.   This follows from \cite{EpsteinMillettTischler}
since leaves with holonomy correspond to fix points for the holonomy pseudogroup, which is countable since we are in the holomorphic category,
and the transversal is one dimensional.  \

Fix a leaf $L$ without holonomy in $Y$ and a point $w\in L$.  By the construction 
there exists a sequence of points $w_j\in X\setminus E$ such that $w_j\rightarrow w$
and $\beta(w_j)\rightarrow 1$ as $j\rightarrow\infty$.    By the upper semi-continuity of $\beta$
we have that $\beta(w)=1$.  Hence all leaves without holonomy in $Y$ (resp. in Supp($T$))
are disks.  

\end{proof}
\

\section{Proof of Theorem \ref{dichotomy}}

We will give in this section a stronger version of Theorem \ref{dichotomy}, but first 
we give a proof of Theorem \ref{dichotomy} using Theorem \ref{main}.

\emph{Proof of Theorem \ref{dichotomy}:}
Assume that there exists a leaf $L$ with a universal covering map $f:\triangle\rightarrow L$
such that the limit set $\Lambda(\Gamma_f)\neq b\triangle$.  In that case, we may assume 
that the segment $\{e^{i\theta}, \theta\in I_s\}, I_s=[-s,s],$ does not intersect $\Lambda(\Gamma_f)$
for some $s>0$.
By the proof of Theorem 5.3 in \cite{FS3}
we cannot have that $\lim_{r\rightarrow 1}f(re^{i\theta})\in E$ for all $\theta\in I_s$, and so 
in particular there is a $\theta_0\in I_s$ and $r_j\rightarrow 1$, such that 
$\mathrm{dist}(f(r_j e^{i\theta_0}),E)\geq \epsilon>0$, measured by some given metric on $\mathbb P^n$.
We will show that $\lim_{j\rightarrow\infty}\alpha(f(r_j e^{i\theta_0}))=1$, in which case Theorem \ref{dichotomy}
follows by Theorem \ref{main} and Proposition \ref{implications}.  \

Recall that 
\begin{equation}
\alpha(f(z))=\Pi_{\gamma\neq\mathrm{id}}|\frac{z-\gamma(z)}{1-\overline{\gamma(z)}z}|.
\end{equation}
Now, by \cite{Rao}, (3.8), we have for any $\gamma$ that 
\begin{equation}\label{estterms}
\log|\frac{1-\overline{\gamma(z)}z}{z-\gamma(z)}|\leq\frac{(1-|z|^2)^2(1-|\gamma(0)|^2)}{|\gamma(0)|^2|z-\zeta_1|^2|z-\zeta_2|^2},
\end{equation}
where $\zeta_1$ and $\zeta_2$ are the two fixed points of $\Gamma_f$.  All these fixed points are in $\Lambda(\Gamma_f)$, and 
so 
\begin{equation}
\log|\frac{1-\overline{\gamma(re^{i\theta})}re^{i\theta}}{re^{i\theta}-\gamma(re^{i\theta})}|\leq C(1-r^2)^2(1-|\gamma(0)|^2),
\end{equation}
for all $\theta\in I_s$ and $C>0$ fixed.   That the leaf $L$ carries a Green's function is equivalent to 
$\Gamma_f$ being of convergence type, \emph{i.e.}, we have that 
\begin{equation}
\sum_{\gamma\neq\mathrm{id}}1-|\gamma(0)|<\infty, 
\end{equation}
and so 
\begin{equation}\label{estlog}
-\log\alpha(f(re^{i\theta}))\leq C'(1-r^2) \mbox{ for } \theta\in I_s.
\end{equation}
$\hfill\square$

Before giving a strengthening of Theorem \ref{dichotomy}, we give a corollary to it. 

\begin{corollary}
Let $(X,\mathcal L)$ be a Brody hyperbolic holomorphic foliation on a compact complex manifold $X$
of dimension $d=2$. 
Assume that no 
leaf is compact, and that all singular points are hyperbolic.  
Then if there exists a leaf $L$ of finite genus and with countably many ends,
there is a minimal set $Y\subset X$ such that a generic leaf in $Y$ is a disk.
\end{corollary}
\begin{proof}
By \cite{HeSchramm} such a leaf is biholomorphic to subset of a compact 
Riemann surface, all of whose boundary components are smoothly bounded or  
points.  Since there are only countably many boundary components there 
has to be an isolated component.  This component cannot be a point, because 
there would be arbitrarily Kobayashi-short nontrivial curves, hence 
there is a smoothly bounded isolated boundary component.  This implies 
that the limit set of the group associated to a universal covering map of $L$
is not everything.  
\end{proof}

\begin{theorem}
Let $(X,\mathcal L,E)$ be a Brody hyperbolic holomorphic foliation on a compact complex manifold $X$
of dimension $d=2$. 
Assume that there is no compact leaf and that all singularities are hyperbolic.  
Assume further  that there is a leaf $L$ with universal covering map $f:\triangle\rightarrow L$, 
and a set $F\subset b\triangle$ of positive measure, such that at each 
point $\zeta\in F$ there is a horocycle 
\begin{equation}
D_\zeta=D_r((1-r)\zeta), 0<r<1, r=r(\zeta),
\end{equation}
on 
which $f$ is injective.    Then there is a minimal set $Y\subset X$ such that 
all but countably many leaves in $Y$ are disks. 
\end{theorem}
A consequence of the Theorem is that 
the structure of some "complicated" leaves, imply that generic leaves are discs.
In the final section we will give an example of a Riemann surface $\triangle\overset{f}{\rightarrow}\triangle/\Gamma$
such that $\Lambda(\Gamma)=b\triangle$, but there is a set $F\subset b\triangle$ of full measure, on which there are injective horocycles for $f$.

\begin{proof}
For each $\zeta\in b\triangle$ we set
\begin{equation}
\sigma(\zeta):=\sup\{0<r<1:f \mbox{ is injective on } D_{r}((1-r)\zeta)\}.
\end{equation}
Then $\sigma$ is upper semi-continuous by Hurwitz' Theorem, and so each set
$E_n:=\{\zeta:\sigma(\zeta)\geq 1/n\}$ is measurable and closed.   So there is a set $E_N$
of positive measure.   \

We will now construct a sequence of $\partial\overline\partial$-closed currents $T_j$
and a sequence $r_j\rightarrow 1$ such that $\beta\geq r_j$  at all points on all leaves without 
holonomy in $\mathrm{Supp}(T_j)$.  Then if $T$ is any cluster point of $\{T_j\}$, 
by upper semi-continuity of $\beta$ we have that $\beta\equiv 1$ on all leaves
without holonomy in $\mathrm{Supp}(T)$.  \

\begin{lemma}
Let $\zeta\in E_N$.  Then for all $n>N$ we have for $z\in D_{1/n}((1-1/n)\zeta)$ that
\begin{equation}\label{betaest}
\beta(f(z))\geq\frac{\frac{n(2-1/n)}{N(2-1/N)}-1}{\frac{n(2-1/n)}{N(2-1/N)}+1}=:M(N,n).
\end{equation}
\end{lemma}
\begin{proof}
The Kobayashi distance between $1-1/N$ and $1-1/n$ is 
\begin{equation}
\frac{1}{2}[\log(\frac{2-1/n}{1/n}) - \log(\frac{2-1/N}{1/N})]=\frac{1}{2}\log\frac{n(2-1/n)}{N(2-1/N)}=m(N,n).
\end{equation}
We let $s<1$ approach 1, and choose $r=r(s)$ such that $\phi_r(z)=\frac{z+r}{1+rz}$
maps the point $-s$ to $1-1/N$.  Then $\phi_r(D_s(0))\subset D_{1/N}(1-1/N)$.
If we choose $\tilde s<s$ such that the Kobayashi distance between $\tilde s$ and $s$
is $m(N,n)$ then $\phi_r(-\tilde s)=1-1/n$ and any point in $D_{1/n}(1-1/n)$
is eventually contained in $\phi_r(D_{\tilde s}(0))$ as $s\rightarrow 1$.  This shows
that the Kobayashi distance from any point in $D_{1/n}(1-1/n)$ to 
the complement of $D_{1/N}(1-1/N)$ is at least $m(N,n)$.  By the injectivity 
of $f$ on $D_{1/N}(1-1/N)$ this means that $\Gamma$ cannot identify 
a point $z\in D_{1/n}(1-1/n)$ with a point closer too it than $m(N,n)$.
Choosing another universal covering map $\tilde f$ with 
$\tilde f(0)=z$ this means precisely \eqref{betaest}. 
\end{proof}

Now for each $n>N$ we define 
\begin{equation}
U_n:=(\cup_{\zeta\in E_N}D_{1-1/n}((1-1/n)\zeta))\cup D_{1-1/n}(0).
\end{equation}
Then for each $n$ the domain $U_n$ is simply connected, and we let $\varphi_n:\triangle\rightarrow U_n$
be a Riemann map with $\varphi(0)=0$.  Set $f_n:=f\circ\varphi_n$ and $G_n=\log\frac{1}{|\varphi_n^{-1}|}$.
Now for $0<r<1$ we define a (1,1)-current $G_{r}$ on $\triangle$ by
\begin{equation}
\langle G_r,\omega\rangle:=\int_{\triangle}\log^+(\frac{r}{|\zeta|})\omega,
\end{equation}
where $\omega$ is a $(1,1)$-form, and then $T_{n,r}:=(f_n)_*G_r$.  As before we want to show that $\|T_{n,r}\|\rightarrow\infty$
as $r\rightarrow 1$.  In that case any cluster point $T_n$ of $\{\frac{T_{n,r}}{\|T_{n,r}\|}\}$ as $r\rightarrow 1$
will be $\partial\overline\partial$-closed, and so its support is a Riemann surface lamination $X_n\subset X$.  By the 
previous lemma and upper semi-continuity of $\beta$ on all leaves without holonomy in $X_n$, 
we have that $\beta\geq M(N,n)$ on these leaves.   Finally we may consider  a cluster point $T$ of $\{T_n\}$.  \

So we fix a positive smooth test form $\omega$ of type $(1,1)$ and estimate 
\begin{align*}
\langle T_{n,r},\omega\rangle & =\int_\triangle\log^+(\frac{r}{|\zeta|})f_n^*\omega=\int_\triangle (\varphi_n)_*(\log^+(\frac{r}{|\zeta|})f_n^*\omega)\\
& = \int_{U_n}\log^+(\frac{r}{|\varphi_n^{-1}(z)|})f^*\omega.
\end{align*}
To get the unbounded mass we need to show that 
\begin{equation}
\int_{U_n}\log|\frac{1}{\varphi_n^{-1}(z)}|f^*\omega = \int_{U_n} G_n(z)f^*\omega \sim \int_{U_n}G_n(z)|f'(z)|^2dV = \infty,
\end{equation}
where $G_n$ is the Green's function on $U_n$.  Now by \cite{FS2} there exists a constant $c_0>0$
such that $|f'(se^{i\theta})|\leq\frac{c_0}{1-s}$ for all $se^{i\theta}$ and for each 
compact set $K\subset X\setminus E$ there is a constant $c_K$ such that $|f'(se^{i\theta})|\geq\frac{c_K}{1-s}$
when $f(se^{i\theta})\in K$.

Set 
\begin{equation}
A:=\{e^{i\theta}:\lim_{s\rightarrow 1} f(se^{i\theta})\in E\}.
\end{equation}
By \cite{FS3} page 951 we have that $A$ has measure zero, and so there exists
a set $\tilde A\subset E_N$ of positive measure and $\epsilon>0$, such that for all $e^{i\theta}\in\tilde A$, 
we have that $\mathrm{dist}(f(se^{i\theta}),E)>\epsilon$ for infinitely many $s\rightarrow 1$.

By Harnack's inequality there exists a constant 
$c_1>0$ such that for all $\zeta\in E_N$ we have that $G_n(s\zeta)\geq c_1\cdot (1-s)$.   By Fubini's Theorem 
it suffices for us to show that  
\begin{equation}\label{intest}
\int_0^1 (1-s)|f'(se^{i\theta})|^2ds=\infty,
\end{equation}
for $e^{i\theta}\in\tilde A$.  Fix $c_2>0$ such that $|f'(\zeta)|\geq\frac{c_2}{1-|\zeta|}$ for 
all $\zeta$ with $\mathrm{dist}(f(\zeta),E)\geq\epsilon/2$.   Then 
if $\mathrm{dist}(f(se^{i\theta}),E)\geq\epsilon/2$ for $s\in [a,b]$ we have that 
\begin{equation}\label{interval}
\int_a^b (1-s)|f'(se^{i\theta})|^2ds\geq \int_a^b\frac{c_2}{1-s}ds.
\end{equation}
Fix $e^{i\theta}\in\tilde A$.  
If $\mathrm{dist}(f(se^{i\theta}),E)\geq\epsilon/2$ for all $s$ close enough to one, it 
is clear that \eqref{intest} is infinite.  So we have to consider the case where 
$\mathrm{dist}(f(se^{i\theta}),E)<\epsilon/2$ for infinitely many $s$.  Then 
we may choose a sequence $a_1<b_1<a_2<b_2<\cdot\cdot\cdot<a_j<b_j<\cdot\cdot$, 
such that $\mathrm{dist}(f(a_je^{i\theta}),E)<\epsilon/2$ and $\mathrm{dist}(f(b_je^{i\theta}),E)\geq\epsilon$.
So 
\begin{equation}
 \sum_{j=1}^\infty\int_{a_j}^{b_j}\frac{1}{1-s}ds=\infty.
\end{equation}

\end{proof}

\section{Examples and applications}

\begin{example}\label{exdisks}
We will construct a hyperbolic Riemann surface lamination such that the generic leaf is the disk, while 
the rest, countably many, are annuli.  Let $\Gamma$ be a Fuchsian group such that $\triangle/\Gamma$
is a compact Riemann surface of genus greater than one, and let $\iota:\Gamma\rightarrow\mathrm{Diff}(S^1)$
be a representation of $\Gamma$ given by the identity map, \emph{i.e.}, we simply restrict $\Gamma$ to $S^1$.
Let $X$ be the quotient of $\triangle\times S^1$ by the group consisting of elements $(\gamma,\iota(\gamma))$
with $\gamma\in\Gamma$.  Fix $s\in S^1$ and 
consider the image of $\triangle_s=\triangle\times\{s\}$ in $X$.  For two points $(\zeta_1,s)$ and $(\zeta_2,s)$
to be identified in $X$ we need $\gamma\in\Gamma$ such that $\gamma(\zeta_1)=\zeta_2$ and $\gamma(s)=s$. 
So $\triangle_s$ is mapped injectively into $X$ for all wandering points $s$.  These are all but 
countably many points.   For a periodic point 
$s$ we let $\Gamma_s$ be the isotropy subgroup of $\{s\}$, and we get that $\triangle/\Gamma_s$ 
injects into $X$.  That the rest of the leafs are annuli follows from the following lemma. 
\begin{lemma}
Let $\Gamma$ be a hyperbolic Fuchsian group, \emph{i.e.}, all elements have 
two fixed points on $b\triangle$, and let $s\in b\triangle$.  Then 
the isotropy group $\Gamma_s$ is either empty or generated by a single element.  
\end{lemma}
\begin{proof}
Assume that $\Gamma_s$ is not empty, and fix $\gamma_1\in\Gamma_s$.  After conjugation, we 
may assume that $s=-1$, and that the other fixed point of $\gamma_1$ is $1$.   We first claim 
that for any other $\gamma_2\in\Gamma_s$ the point $1$ is also a fixed point.   If not, let 
$p\notin\{-1,1\}$ be a fixed point for $\gamma_2$.  Assuming that $1$ is attracting 
for $\gamma_1$ we set $p_j:=\gamma_1^j(p)$ to obtain a sequence of points 
$p_j$ converging to $1$.   Set $\sigma_j:=\gamma_1^j\circ\gamma_2\circ\gamma_1^{-j}$.
Then $\sigma_j$ has fixed points $-1$ and $p_j$, and all maps have the same multipliers 
$\pm\lambda$, those of $\gamma_2$.  So the sequence $\sigma_j$ converges to the map 
fixing $\pm 1$ and with multipliers $\pm\lambda$.  But since $p_j$ is never equal to 
one this contradicts the discreteness of $\Gamma$.  \

So all elements of $\Gamma_s$ is of the form
\begin{equation}
\gamma_r(z)=\frac{z+r}{1+rz}, 
\end{equation}
with $r$ real.  Discreteness of $\Gamma$ implies that there is a smallest 
positive $r$ such that $\gamma_r$ is in the group, and we claim that this element 
generates $\Gamma_s$.  Assume to get a contradiction that $\sigma\in\Gamma_s$
is not in $[\gamma_r]$, and that $\sigma(0)>0$.  Since $\sigma(0)>\gamma_r(0)$
there exists a largest integer $k>0$ such that $\gamma_r^k(0)<\sigma(0)$.  
Then 
\begin{equation}
\sigma^{-1}(\gamma_r^{k+1}(0))=\mathrm{d_M}(\sigma(0),\gamma^{k+1}_r(0))<\mathrm{d_M}(\gamma_r^k(0),\gamma_r^{k+1}(0))=\gamma_r(0),
\end{equation}
which contradicts the minimality of $\gamma_r$.

\end{proof}
\end{example}

Although, as we just have seen, there exist hyperbolic laminations for which there is a countable dense set of leaves which are annuli, 
the following shows that the set of annuli cannot be too large: 

\begin{proposition}
Let $(X,\mathcal L)$ be a minimal compact foliated metric space, such that a residual set $S$ of leaves 
have finite genus and at most countably many ends.   
Then if a generic leaf is not a topological disk, there is no leafwise complex structure on $X$ such that all leaves are hyperbolic.  
\end{proposition}
\begin{proof}
Assume that there is a hyperbolic structure on $X$, and let $L$ be a leaf of finite genus
and at most countably many ends.  Then by \cite{HeSchramm} $L$ is isomorphic to 
to a domain $\Omega\subset Y$, where $Y$ is a compact Riemann surface, and 
all boundary components of $b\Omega$ are smooth or points.  Since there are 
only countably many boundary components, there is an isolated component, and 
since $(X,\mathcal L)$ is non-singular, this component cannot be a point, since we would 
find arbitrarily Kobayashi short non-trivial loops in $L$.  So the limit set $\Lambda(\Gamma)$
of the group $\Gamma$ associated to $L$ is different from $b\triangle$, and so by 
Theorem \ref{main2}, the remark following it, and the proof of Theorem \ref{dichotomy},
the generic leaf is a disk.  A contradiction.

\end{proof}
By considering for instance suspensions over tori, there are Riemann surface laminations all of whose leaves 
are topological cylinders, in this case all of them conformally isomorphic to $\mathbb C^*$; the point is that 
you can never give such a foliation a hyperbolic structure.

\begin{example}
Let $\triangle/\Gamma$ be a compact Riemann surface of genus two.  Then $\Gamma$ has four 
generators $a_1,b_1,a_2,b_2$, and we have the relation 
\begin{equation}\label{rel}
a_1b_1a_1^{-1}b_1^{-1}a_2b_2a_2^{-1}b_2^{-1}=\mathrm{id},
\end{equation}
and \eqref{rel} is the only relation.  This means that we may define a group 
homomorphism $\phi:\Gamma\rightarrow\Gamma$ by 
\begin{equation}
\phi(a_1)=a_1, \phi(b_1)=a_1, \phi(a_2)=a_2, \mbox{ and } \phi(b_2)=a_2.
\end{equation}
Now let $\tilde\Gamma$ be the group consisting of elements $(\gamma,\phi(\gamma)), \gamma\in\Gamma,$
acting on $\triangle\times S^1$.  Let $X=(\triangle\times S^1)/\tilde\Gamma$ denote the quotient with 
a natural projection $\pi:X\rightarrow Y =\triangle/\Gamma$.  The leaves of the foliation on $X$
are the images of disks $\triangle_s=\triangle\times\{s\}$ via the quotient map defined by $\tilde\Gamma$.  
As in Example \ref{exdisks} the image of $\triangle_s$ is biholomorphic to $\triangle/\Gamma_s$
where $\Gamma_s$ is the stabiliser 
\begin{equation}
\Gamma_s:=\{\gamma\in\Gamma:\phi(\gamma)(s)=s\}.
\end{equation}
For all but countably many $s$ we have that $\Gamma_s$ is simply $\mathrm{Ker}(\phi)$,
this is the set of wandering points $s$.  So the generic leaf is not only of a fixed
topological type; all but countably many leaves have the same conformal type $\triangle/\Gamma_s$.   In fact, 
if $T$ is a transversal and $L_1$ and $L_2$ are two such leaves passing through 
points $t_1,t_2\in T$, there is a biholomorphism $g:L_1\rightarrow L_2$ with 
$g(t_1)=t_2$.  This means that for a point $y\in Y$, the functions $\alpha,\beta$ and $\rho$
are constant on the intersection of $\pi^{-1}(y)$ with the set $\mathcal G$ of generic leaves, \emph{i.e.},
on $\mathcal G$ the functions $\alpha,\beta$ and $\rho$ are functions on $Y$, and there
they are continuous.   So both their suprema and infima will actually be reached on $\mathcal G$, 
and it is clear that the supremum is not one.  Moreover, since $\mathrm{Ker}(\phi)\subset\Gamma_s$
for all $s$, the supremum is in fact a strict supremum on $X$.  
\end{example}

\begin{example}\label{hypsing}
Let $(X,\mathcal L,E)$ be a holomorphic foliation on $\mathbb P^2$ with hyperbolic singularities. 
Then for any point $e\in E$ there exists a sequence $z_j\rightarrow e$ such that $\rho(z_j)\rightarrow 1$.
For simplicity we assume that near a point $e$ the foliation is defined by the vector field 
$X(z)=\frac{\partial}{\partial z_1} + \lambda\frac{\partial}{\partial z_2}$, so 
that integral curves are given by $\phi_{(x_0,y_0)}(z)=(x_0e^z,y_0e^{\lambda z})$.
To construct large injective Kobayashi disks, we consider first 
annuli in the separatrix $\{z_2=0\}$.   For any $M>0$ there exist 
real numbers $a<x_0<b$, with $b$ arbitrarily small, 
such that any curve $\gamma\subset A(a,b)$ connecting 
$x_0$ and $bA(a,b)$ has Kobayashi length at least $M$.  Furthermore, 
we may choose an integer $N$, such that any curve starting at $x_0$
and circling the origin $N$ times has Kobayashi length at least $M$.  \

Now consider 
\begin{equation}
S=\{z\in\mathbb C:\log a<\mathrm{Re}(z)<\log b, -2\pi N<\mathrm{Im}(z)<2\pi N\}.
\end{equation}
If $y_0>0$ is chosen small enough, then $\phi_{(1,y_0)}:S\rightarrow X$
is an injective parametrisation of a piece of a leaf, and there 
is a $2N-1$ covering map $\pi_1:\phi_{1,y_0}(S)\rightarrow A(a,b)$.  Now 
if $y_0$ is small enough, the Kobayashi length of a curve $\gamma$ in $\phi_{x_0,y_0}(S)$
is roughly the same as $\pi_1(\gamma)$, by the continuity of the Kobayashi metric. 
If a loop $\gamma$ based at $(x_0,y_0)$ is non-trivial, then either $\pi_1(\gamma)$
will have to leave $A(a,b)$, or it has to circle the origin at least $N$ times.   So 
its Kobayashi length is at least $M$.  

\end{example}

\begin{example}\label{fullmeasure}
We will construct a Fuchsian group $\Gamma$ such that the following holds:
\begin{itemize}
\item[(i)] $\Lambda(\Gamma)=b\triangle$, 
\item[(ii)] $L=\triangle/\Gamma$ has a Cantor set of ends, 
\item[(iii)] There is a set $F\subset b\triangle$ of full measure, such that 
for each $\zeta\in F$ there is a horocycle $D_{r}((1-r)\zeta)$ on which the universal 
covering map $f:\triangle\rightarrow L$ is injective. 
\end{itemize}
In particular, if $L$ were a leaf of a generic foliation on $\mathbb P^2$, 
the generic leaf of the unique minimal set would be biholomorphic to the unit disk.  \

The construction will be along the lines described in the next section.  Let $\gamma_1$
by an element of the form \eqref{hyperbolic} and set $\Gamma_1=[\gamma_1]$.
The limit set $\Lambda(\Gamma_1)$ consists of two points, and we may choose 
a closed set $F_1\subset b\triangle\setminus\Lambda(\Gamma_1)$ of 
length $2\pi-1$ and an $r_1>0$ such that for each point $\zeta\in F^1_1$, 
the group $\Gamma_1$ does not identify points on the horocycle
$D_r((1-r)\zeta)$. \

To construct $\gamma_2$ we write $b\Omega_1\cap b\triangle=c_1^1\cup c_1^2$, a
union of two arcs.  Choose $a_1\in c^1_1$ such that $a_1$ divides $c^1_1$ into 
two pieces of the same length.  For any $\epsilon_2>0$ there exists $\gamma_2$
with a fundamental domain $\Omega_2$ such that $D_2=\triangle\setminus\overline{\Omega_2}\subset D_{\epsilon_2}(a_1)$.
Set $\Gamma_2=[\Gamma_1,\gamma_2]$.  Then $\Gamma_2$ makes no further identifications 
on $\triangle\setminus D_{\epsilon_2}(a_1)$, so for any $\delta_2>0$ we may choose 
$\epsilon_2$ small enough such that there exists $F_1^2\subset F_1^1$ of 
length $2\pi-1-\delta_2$, such that for each point $\zeta\in F^2_1$, 
the group $\Gamma_2$ does not identify points on the horocycle
$D_{r_1}((1-r_1)\zeta)$.   Now $\Lambda(\Gamma_2)$ has measure zero (see \emph{e.g.} \cite{Beardon}, Theorem 4), 
so there exists a closed set $F_2^1\subset b\triangle$ of length $2\pi-1/2$ not intersecting $\Lambda(\Gamma_2)$, 
and $r_2>0$ such that for each point $\zeta\in F^1_2$, 
the group $\Gamma_2$ does not identify points on the horocycle
$D_{r_2}((1-r_2)\zeta)$. \

To construct $\gamma_3$ we consider a point $a_2\in c^1_2$, and repeat the argument
to find $\gamma_3$, arbitrarily small $\delta_3>0$, and $F_1^3\subset F_1^2, F_2^2\subset F_2^1$, 
of length $2\pi-1-\delta_2-\delta_3$ and $2\pi-1/2-\delta_3$ respectively, 
such that on $F_1^3$ one finds injective horocycles of radius $r_1$, and on $F_2^2$
one finds injective horocycles of radius $r_2$.   Set $\Gamma_3=[\Gamma_2,\gamma_3]$.
Again $\Lambda(\Gamma_3)$ has measure zero, so there exists a closed set 
$F_3^1\subset\triangle\setminus\Lambda(\Gamma_3)$ of length $2\pi-1/3$ on 
which we can find injective horocycles of radius $r_3>0$ for some $r_3>0$.  \

At this point it is clear how to continue constructing $\gamma_j,\delta_j,F_i^l,r_j$ such 
that for each group $\Gamma_m=\{\gamma_1,...\gamma_m\}$ with fundamental 
domain $\Omega_m$, we have that 
\begin{itemize}
\item[(i)] $F_i^{m-i+1}$ has length $2\pi-1/i - \sum_{j=i-1}^m\delta_j$,
\item[(ii)] on each $F_i^{m-i+1}$ there are injective horocycles of radius $r_i$ for the group $\Gamma_m$.
\item[(iii)] $b(\cap_m\Omega_m)$ contains no open interval on the intersection with $b\triangle$.
\end{itemize}

\end{example}

\section{A construction of infinite Fuchsian Groups}

In this section we will describe a general inductive construction of an infinite Fuchsian group. \

Recall that a fundamental domain for $\Gamma$ is an open set $\Omega\subset\triangle$
such that all points in $\triangle$ has its equivalent in $\overline{\Omega}$, and no two points
in $\Omega$ are equivalent.   A standard fundamental domain for $\Gamma$ is 
\begin{equation}
\Omega:=\{z\in\triangle:[z,0]<[z,\gamma(z)] \mbox{ for all } \gamma\neq\mathrm{id}\}.
\end{equation}
Here $[\cdot,\cdot]$ denotes the M\"{o}bius distance.  \
  
Recall that an element $\gamma\in\mathrm{Aut_{hol}}\triangle$   
is hyperbolic if it has exactly two fixed points on $b\triangle$.   Our basic example 
is a mapping 
\begin{equation}\label{hyperbolic}
\gamma(\zeta)=\frac{\zeta+r}{1+r\zeta},
\end{equation}
and it is easy to see that any hyperbolic element is conjugate to a mapping 
on the form \eqref{hyperbolic}. \

If we let $l_{+}$
denote the geodesic connecting $e^{i(\frac{\pi}{2}-\arcsin r)}$
and $e^{i(\arcsin r - \frac{\pi}{2})}$, and let $l_{-}$
denote the geodesic connecting$e^{i(\frac{\pi}{2}+\arcsin r)}$
and $e^{i(-\arcsin r - \frac{\pi}{2})}$,  the fundamental 
domain $\Omega$ for $\Gamma=[\gamma]$ is the domain in $\triangle$
bounded by $l_+$ and $l_-$.  The lines $l_{\pm}$ are
geodesics passing through the points $\frac{1-\sqrt{1-r^2}}{r}$
and $\frac{\sqrt{1-r^2}-1}{r}$ on the real line, and the two lines 
are identified by $\gamma_1$, making $\triangle/\Gamma$
an annulus.  Denote by $D_1$ and $D_2$
the two connected components of $\triangle\setminus\overline{\Omega}$. \

We may now describe an inductive procedure to construct an infinite Fuchsian group. 
Assume that we have constructed hyperbolic elements $\gamma_1,...,\gamma_n$,
such that $\Gamma_m=[\gamma_1,...,\gamma_m]$ is a Fuchsian group with 
fundamental domain $\Omega_m$, such that $b\Omega_m\cap b\triangle$
consists of a union of arcs, \emph{i.e.}, the underlying Riemann surface $L_m=\triangle/\Gamma_m$
is a bordered Riemann surface.   Let $a$ be an interior point of one of the boundary arcs in 
$b\Omega_m\cap b\triangle$, and choose $0<\epsilon_m<<1$.  It is clear that there 
exists a M\"{o}bius transformation $\phi(z)$ such that if we define $\gamma_{m+1}(z)=\phi^{-1}(\gamma(\phi(z)))$,
then $[\gamma_{m+1}]$ has a fundamental domain $U_{m+1}$ with $\triangle\setminus\overline{U_{m+1}}\subset D_{\epsilon_m}(a)$.  It follows from Lemma \ref{limitgroup} that if $\epsilon_m$ is sufficiently small, then $[\Gamma_m,\gamma_{m+1}]$
is a Fuchsian group with fundamental domain $\Omega_m\cap U_{m+1}$.   Set $\Gamma=[\gamma_1,\gamma_2,...]$
for a sequence defined inductively like this.  If $\epsilon_m\searrow 0$ sufficiently fast, it is easy to see that 
$\Gamma(\overline\Omega)=\triangle$, where $\Omega=\cap_m\Omega_m$, since $\Gamma_m(\overline\Omega_m)=\triangle$,
and so the argument in the proof of Lemma \ref{limitgroup} gives the discreteness of $\Gamma$.

\begin{lemma}\label{limitgroup}
Let $\Gamma$ be a finitely generated hyperbolic Fuchsian group with fundamental 
domain $\Omega$.  Furthermore, let $\gamma\in\mathrm{Aut_{hol}}\triangle$ 
by a hyperbolic element with a fundamental domain $U$ such that $D=\triangle\setminus U\subset\Omega$, the boundary
$bU$ is contained in $\Omega$, and such that $bU\cap\triangle$ is bounded away 
from $b\Omega\cap\triangle$.
Then $\tilde\Gamma=[\Gamma,\gamma]$ is a hyperbolic Fuchsian group, with 
a fundamental domain $\tilde\Omega=\Omega\cap U$.
\end{lemma}
\begin{proof}
We will first show that if $p\in\overline{\tilde\Omega}$ and if 
$\gamma(p)=p$ for $\gamma\in\tilde\Gamma$, then $\gamma=\mathrm{id}$.
Assume first that $p\in b\tilde\Omega$, and furthermore that $p\in b\Omega$.
Choose $\delta>0$ such that $\Gamma(B_\delta(p))\cap\overline D=\emptyset$.
Write 
\begin{equation}
\gamma_k\circ\gamma_{k-1}\circ\cdot\cdot\cdot\circ\gamma_1,
\end{equation}
where the $\gamma_j$'s alter in belonging to either $\Gamma$ or $[\gamma]$, 
and non of them are the identity map. 
Assume first that $\gamma_1\in\Gamma$.  Then by our assumption above, 
we have that $\gamma_2(\gamma_1(B_\delta(p)))\subset D$.   We now prove by 
induction that for any $m\geq 1$ we have that $\gamma_{2m}\circ\cdot\cdot\cdot\circ\gamma_1(B_\delta(p))\subset D$, 
and $\gamma_{2m+1}\circ\cdot\cdot\cdot\circ\gamma_1(B_\delta(p))\subset\triangle\setminus\overline\Omega$.
If the first claim holds for some $m$ then clearly the second claim holds, since a non-trivial 
element of $\Gamma$ will map $D$ outside $\Omega$, $\Omega$ being a fundamental domain 
for $\Gamma$ and $\overline D\subset\Omega$.  By the same reasoning, if the second 
claim holds for some $m$, then the first claim holds for $m+1$.   So $\gamma$ cannot 
have a fixed point.  Similar arguments work if $p\in\tilde\Omega$ or $p\in bU$. \

Next we let $p\in\triangle$.  We now show that if $\gamma_j\in\tilde\Omega$ for $j=1,2,...,$
and if $\gamma_j(p)\rightarrow q\in\overline{\tilde\Omega}$, then 
$\gamma_j=\mathrm{id}$ for $j\geq N$, for some $N>0$.   \
Assume that $q\in b\Omega$, the case $q\in bU$ will be completely analagous.   
Choose $\delta>0$ such that $\Gamma(B_\delta(q))\cap\overline D=\emptyset$.
We will consider the maps $\alpha_j=\gamma_{j+1}\gamma_j^{-1}$, 
such that, setting $q_j=\gamma_j(p)$, we have $q_j\rightarrow q$ and 
$\alpha_j(q_j)=q_{j+1}$.  Now if $\alpha_j$ is eventually 
in $\Gamma$ for large $j$, then $\alpha_j=\mathrm{id}$ for large $j$.  
So $\gamma_j=\gamma_{j+1}$ for large $j$, hence $\gamma_j(q)=q$
for large $j$, and so $\gamma_j=\mathrm{id}$ for large $j$, since
no other element of $[\Gamma,\gamma]$ can fix an element of $\overline{\tilde\Omega}$. \

We may finally show that $\tilde\Gamma$ is discrete.  Assume that $\gamma_j\in\tilde\Gamma$
and $\gamma_j\rightarrow\gamma\in\mathrm{Aut_{hol}}\triangle$.  Then 
$\gamma_j(0)\rightarrow\gamma(0)$.  By the lemma below there exists 
$\phi\in\tilde\Gamma$ such that $\phi(\gamma(0))\in\overline{\tilde\Omega}$, 
and so $\phi\circ\gamma_j(0)\rightarrow q\in\overline{\tilde\Omega}$.  So 
$\phi\circ\gamma_j=\mathrm{id}$ for large $j$.

\end{proof}

\begin{lemma}
Let $\Gamma$ be a finitely generated hyperbolic Fuchsian group with fundamental 
domain $\Omega$.  Furthermore, let $\gamma$ by a hyperbolic element with 
a fundamental domain $U$ such that the complement $\triangle\setminus\overline U$
is contained in $\Omega$.  Then $[\Gamma,\gamma](\overline{\Omega'})=\triangle$, 
where $\Omega'=\Omega\cap U$.
\end{lemma}
\begin{proof}
Let $A_1,A_2$ denote the two components of $\triangle\setminus\overline U$; these 
are the intersection of $\triangle$ and two disjoint euclidean disks $D_1$ and $D_2$
in $\mathbb C$.  It suffices to show that $U\subset [\Gamma,\gamma](\overline\Omega)$. \
We consider what we can reach with compositions 
\begin{equation}
\gamma_k\circ\gamma_{k-1}\circ\cdot\cdot\cdot\circ\gamma_1,
\end{equation}
where the $\gamma_j$'s alter in belonging to $\Gamma$ and $[\gamma]$, 
and $\gamma_1\in\Gamma$, and we let $\mathcal A_k$ 
denote the set of points we can reach with a composition of length $k$.  
Note first that $\mathcal A_1$ is all of $U$ except the full orbit $\mathcal F_1=\Gamma(A)$
where $A=A_1\cup A_2$.  Next, by composing by elements of $[\gamma]$ 
we can reach all points in $A$ except the full orbit $\mathcal F_2=[\gamma](\mathcal F_1)$, but no 
additional points in $U$.  But now $\mathcal F_2$ is a countable family of pieces of Euclidean disks 
contained in $A$, and $\mathcal A_3$ will consist of all points in $U$ except 
the full orbit $\mathcal F_3=\Gamma(\mathcal F_2)$.   Continuing in 
this fashion we get a family $\mathcal F_{2i-1}$ of open sets, $\mathcal F_{2i+1}\subset\mathcal F_{2i-1}$, 
each $\mathcal F_{2i-1}$ is a countable family of pieces of Euclidean disk, each disk 
splitting into a countable family of smaller disks in the next $\mathcal F_{2i+1}$.  \

Next, assume to get a contradiction that $\cap_i\mathcal F_{2i-1}$ is not empty.  
This means that there exists a sequence 
\begin{equation}
\alpha_j:=\gamma_{2j-1}\circ\gamma_{2j-2}\circ\cdot\cdot\cdot\circ\gamma_1,
\end{equation}
such that the images $\alpha_j(A_1)$ (or $A_2$) decrease to a nontrivial intersection 
of $\triangle$ with a a Euclidean disk 
$D_3\subset\mathbb C$.   We may now assume that $\alpha_j\rightarrow\alpha:D_1\rightarrow D_3$
uniformly on compacts.   Now $\alpha$ cannot be constant, otherwise $\alpha_j(A_1)$
could not decrease to a nontrivial intersection with $\triangle$, hence $\alpha$ is a
biholomorphism.  Then for any composition $\beta:=\tilde\gamma_2\tilde\gamma_1$
with $\tilde\gamma_1\in\Gamma$ and $\tilde\gamma_2\in [\gamma]$ such that 
$\beta(D_1)\subset\subset D_1$ we get that $\alpha(\beta(D_1))\subset\subset D_3$, 
and so for a large enough $j$ we have that $\alpha_{j}(\beta(D_1))\subset\subset D_3$.
This is a contradiction.

\end{proof}

\bibliographystyle{amsplain}

\end{document}